\documentclass[12pt,reqno]{amsart}
\usepackage{amsmath,amsfonts,amsthm,amssymb,color}
\usepackage[T1]{fontenc}
\usepackage{enumitem}
\usepackage{hyperref}
\hypersetup{
     colorlinks   = true,
     citecolor    = blue,
     linkcolor    = blue
}
\usepackage[left=1.1in, right=1.1in, top=1.1in, bottom=1.1in]{geometry}
\usepackage[font=small]{caption}

\usepackage{pdfsync}

\usepackage{amsmath}
\usepackage{amssymb}
\usepackage{fancyhdr}
\usepackage{graphicx}
\usepackage{mathrsfs}
\usepackage{amsthm}
\usepackage{cleveref}
\usepackage{xcolor}
\usepackage{lastpage}
\usepackage{fancyvrb}
\usepackage{listings}
\usepackage{bbm}

\usepackage{tikz}

\newcommand{\dd}{\mathrm{d}}

\newcommand{\iot}{\int_{0}^{t}}

\newcommand{\E}{\mathbb{E}}
\newcommand{\PP}{\mathbb{P}}
\newcommand{\R}{\mathbb{R}}

\newcommand{\al}{\alpha}
\newcommand{\ep}{\varepsilon}

\newcommand{\ga}{\gamma}
\newcommand{\ka}{\kappa}
\newcommand{\la}{\lambda}

\newcommand{\lc}{\left[}
\newcommand{\rc}{\right]}

\newtheorem{theorem}{Theorem}[section]

\newtheorem{corollary}[theorem]{Corollary}

\newtheorem{hypothesis}[theorem]{Hypothesis}

\newtheorem{proposition}[theorem]{Proposition}

\theoremstyle{remark}
\newtheorem{remark}[theorem]{Remark}

\crefname{corollary}{Corollary}{Corollaries}
\crefname{definition}{Definition}{Definitions}
\crefname{figure}{Figure}{Figures}
\crefname{hypothesis}{Hypothesis}{Hypotheses}
\crefname{lemma}{Lemma}{Lemmas}
\crefname{notation}{Notation}{Notations}
\crefname{proposition}{Proposition}{Propositions}
\crefname{remark}{Remark}{Remarks}
\crefname{section}{Section}{Sections}
\crefname{theorem}{Theorem}{Theorems}

\title[Queuing model in random environment]{Infinite server queues in a \\ random 
fast oscillatory environment}
\author[Y. Liu \and H. Honnappa \and S. Tindel \and N. K. Yip]
{Yiran Liu \and Harsha Honnappa \and Samy Tindel \and  Nung Kwan Yip}

\address{Harsha Honnappa: School of Industrial Engineering, Purdue
University, 315 N. Grant Street, W. Lafayette, IN 47907, USA.}

 \address{Yiran Liu, Samy Tindel, Nung Kwan Yip: Department of Mathematics,
Purdue University,
150 N. University Street,
W. Lafayette, IN 47907,
USA.}


\email{[honnappa, liu387, stindel, yipn]@purdue.edu}

\thanks{S. Tindel is supported by the NSF grant
  DMS-1952966. H. Honnappa is supported by the NSF grant CMMI-1636069.}
  

\begin{document}
\maketitle

\begin{abstract}
In this paper, we consider a $G_t/G_t/\infty$ infinite server queueing model in a random environment. More specifically, the arrival rate in our server is modeled as a highly fluctuating stochastic process, which arguably takes into account some small time scale variations often observed in practice. We show a homogenization property for this system, which yields an approximation by a $M_t/G_t/\infty$ queue with modified parameters.
Our limiting results include the description of the number of active servers, 
the total accumulated input and the solution of the storage equation. 
Hence in the fast oscillatory context under consideration, we show how the queuing system in a random environment can be approximated by a more classical Markovian system.
\end{abstract}

\section{Introduction}
Nonstationary models have been extensively studied in the literature on queues, particularly in the Markovian setting. A typical assumption in this setting is that the arrival and service intensities are deterministic time-varying functions. In Markovian settings, it is also natural to assume that the intensity functions are smooth~\cite{whitt1}. However, in practice, queueing systems are often subject to ``environmental'' noise: for instance, while arrival intensities to call centers and hospitals display time-of-day (or ``diurnal'') effects, the intensity functions also vary based on the day-of-week and seasonal effects. In other queueing systems, particularly those with high intensity arrivals such as computer networks or cloud service systems, there is also intra-day and intra-hour stochastic variation in the intensity process. The performance of these queueing systems is therefore affected by both the smaller time-scale stochastic variations, as well as (relatively) longer time-scale time-of-day effects.

Our objective in this paper is to understand the interplay between short time-scale stochastic fluctuations and long time-scale variations in the model intensities, and the impact these effects have on the computation of system performance metrics. To this end, we study a $G_t/G_t/\infty$ infinite server queueing model imbedded in a random environment. While infinite server queues are an approximation in the real world, they are a useful vehicle to address the questions of interest to us. We assume a doubly stochastic Poisson process (DSPP) traffic model, so that conditional on the stochastic intensity the cumulative number of arrivals in a fixed time interval is Poisson distributed. 
We assume that the stochastic intensity is modeled as $\mu(s) = \Psi(s,Z_s)$, 
where $Z_s$ is an ergodic stochastic process (as we will see later on, a typical example of such process is an Ornstein-Uhlenbeck type process). 
We will develop much of our theory under the special case of a ``separable'' stochastic intensity function, denoted by $\mu^{\ep}$ and defined by
\begin{equation}\label{eq:intensityfunc-mu}
\mu^{\ep} (\dd s) = \la (s) \psi (Z_{s/\ep}) \dd s , \quad \ep > 0, \, s > 0
\end{equation}
where a deterministic function of time $\la (s)$ (modeling time-of-day effects) is multiplied by some positive function of a stochastic process $Z$ (modeling fluctuations). Note the time scale $\ep^{-1} t$ associated with $Z$. The constant $\ep$ in this paper is intended to be a small parameter reflecting the fast oscillatory nature of the fluctuations. Coming to the service model, we consider a general setting where the parameters of the service time distribution functions are assumed to vary temporally with the long time-scale variations in the traffic intensity. An example is given by Pareto service times, with temporally-varying scale coefficients that depend on the arrival epoch.

Performance analysis of non-Markovian queueing models is in general rather difficult. Consequently, we focus on developing stochastic process approximations as the parameter $\ep \to 0$. Our main result in \cref{thm:limit-N-ep-in-D_T} shows that in the limit the $G_t/G_t/\infty$ queue is closely approximated by a $M_t/G_t/\infty$ queue; i.e., an infinite server queue where traffic is modeled by a Poisson process with deterministic time-varying intensity function. Such a limiting procedure is often called `homogenization,' in the sense that the fast oscillating process is averaged out to produce an effective description which is usually much easier to analyze. More precisely, let $N^{\ep}$ be the quenched stochastic process representing the state of the $G_t/G_t/\infty$ queue. Then $N^{\ep}$ converges weakly to a Poisson point process $N$ that is the state of an $M_t/G_t/\infty$ queue. We write this limit as
\begin{equation}\label{eq:limit-N-epsilon-intro}
\{N^{\ep} (t) : t \in [0, T]\} \xrightarrow{(\dd)} \{N(t) : t \in [0, T]\} , ~\mathbb{P}_Z-\text{a.s.}
\end{equation}
where $T>0$ is an arbitrary time horizon and $\mathbb{P}_Z$ denotes the quenched probability for a fixed environment $Z$. In order to better specify a notion of \textit{quenched} in this context, note that there are inherently two stochasticities in our model: one is the random fluctuation for the arrival rate process $\mu^{\ep}$ defined by \eqref{eq:intensityfunc-mu}, and the other is the actual arrival process $N^{\ep}$ given the rate. The notion of ``quenched'' refers to the probability space upon fixing or given one particular random fluctuation in $\mu^{\ep}$. 

The proof of our main result is crucially dependent on the assumption that the short-term stochastic fluctuation model mixes rapidly enough (see \cref{hyp:ergodicity} for a more precise statement). Consequently, the stochastic fluctuations can reach a steady state within the relatively longer time-scale of the time-of-day effects, thereby ``averaging'' out the short time-scale fluctuations. The proof relies on first showing that the mean measure of the queue state satisfies a strong law type limit. Namely it can be seen that the mean measure of the queue state is an additive functional of the stochastic fluctuations, and when the latter reach steady state fast enough, the mean measure converges to a deterministic limit. This type of result is well known for general Markov processes~\cite{KontoyannisMeyn}. Our homogenization limit follows by leveraging the convergence of mean measure to show that the finite dimensional distributions of the Poisson random measure corresponding to the (quenched) state $N^{\ep}$ of the $G_t/G_t/\infty$ queue converge to those of the state $N$ of an $M_t/G_t/\infty$ queue. A further tightness condition leads to the convergence of the whole process. As the reader can see, the main technical novelty in our paper consists in combining those homogenization results with some more standard considerations about limits for queues.

Our main result \cref{thm:limit-N-ep-in-D_T} yields two crucial insights into the performance analysis of queueing models imbedded in random environments. First, it shows that ergodic properties of the underlying stochastic fluctuations play a critical role in separating the fast stochastic fluctuations from the longer time-scale temporal variations. Second, it shows that under our homogenization limit, it is possible to safely ignore the stochastic fluctuations and model the system using an $M_t/G_t/\infty$ queueing model. The latter model has been extensively studied~\cite{EMW1,EMW2,MW1} and there is a substantial literature available on its properties, particularly with stationary service. From a practical perspective, the $M_t/G/\infty$ queue is also significantly easier to use in simulation studies.

The aforementioned homogenization phenomenon put forward in \cref{thm:limit-N-ep-in-D_T} is similar to the ``rapid fluctuation'' analysis in~\cite{ZHG1}. In that work, the weak convergence of a general point process (for example, a DSPP) to a constant rate Poisson process (under the assumption that the compensator of the point process satisfies a strong law) was used to approximate the state distribution of a $G_t/G/\infty$ queue by that of an $M_t/G/\infty$ queue. The analysis in~\cite{ZHG1} crucially used Taylor expansions of the state probability distribution at a fixed time $t$ in terms of the scaling parameter. The approach in the current paper is completely different. Indeed, we mostly exploit the Poisson random measure representation of the state process, and then establish the process-level stochastic approximation limit for this object. Our paper is also closely related to~\cite{RR}, where an infinite server queue with ``extremely'' heavy-tailed Pareto service times is studied in a time homogeneous setting in order to explain network self-similarity effects. The relatively simple homogeneous setting allowed the authors to establish not only a functional strong law of large numbers (FSLLN), but also a functional central limit theorem (FCLT). Our results in this paper substantially generalize the FSLLN result to a queueing model imbedded in a random environment. 
However, the analysis to establish the corresponding FCLT in our setting is significantly complicated; see our conclusion section.

Our paper can be related to multiple threads of research on nonstationary queueing models and queues embedded in random environments. First, there is a significant body of work developing both uniform acceleration (\cite{MMR,honnappa,MM,SS}) and many-server heavy-traffic limit theorems (\cite{lu1,lu2,CH1}) to time-varying queues. In much of this literature, the limit processes are shown to be (reflected) diffusion processes, where the time-of-day effects manifest themselves as the drift function of the diffusion process. Note that all of this work assumes that the nonstationarity manifests as a deterministic temporal variation. There is also a growing body of work developing asymptotic expansions (\cite{ZHG1,ZHG2,MW2,pender1,pender2,KYZ1,KYZ2}) of performance metrics. It is well known, however, that traffic arriving at call centers and hospitals displays significant over-dispersion relative to a Poisson process with deterministic intensity~\cite{KW1}, implying that a DSPP is an appropriate model of the traffic in these systems. There is a significant literature on queues in random environments. Much of this literature assumes that either the traffic and/or service processes are Markov modulated, where the underlying stochastic environment process is a finite state Markov chain; the vast majority of the related literature focuses on characterizing stationary behavior, but~\cite{CMRW,pender1} exhibit a couple of examples where asymptotic limit theorems and expansions can be established. 

Most relevant to our current setting is the literature on infinite server queues in random environments~\cite{OP,Rol1,Rol2,HvLM,HM,pender2}. While we cannot do a full review of this literature here, we point out, in particular,~\cite{HvLM} where the effect of over-dispersed traffic on the performance of an infinite server queue is studied. Paralleling our findings, this paper shows that a sufficiently rapidly fluctuating environment (relative to a slowly changing arrival intensity) will, in an appropriate asymptotic regime, ensure that the infinite server system behaves like a ``standard'' infinite server queue in steady-state. On the other hand, in~\cite{HvLM} the traffic intensity does not have an explicit time-of-day component and, for analytical reasons, 
the random environment is formulated in a somewhat ``stylized'' fashion.
Our statements complement these results, and more significantly, 
show that the standard infinite server queue behavior is preserved even with explicit time-of-day effects in the traffic and service processes. In addition, let us observe again that our result is obtained in the so-called quenched regime (as opposed to the annealed regime of~\cite{HvLM}). Otherwise stated, the main limit result~\eqref{eq:limit-N-epsilon-intro} is valid for almost any realization of the environment $Z$. This is usually believed to correspond to real world situations, where only a single realization of $Z$ is observed.

The rest of this paper is organized as follows. \cref{sect:notation-model} introduces
the notation that will be used throughout this paper and 
constructs the random arrival model of interest with appropriate hypotheses. 
\cref{sect:law-N-ep} gives preliminary limiting results concerning the mean 
$m^{\ep}$ of the Poisson random variable $N^{\ep}$. We present the main results for the homogenized process $N^{\ep}$ in \cref{sect:homogenized-process}, where we prove the  convergence in law of $N^{\ep}$ as a process. Finally, we end in Section~\ref{sect:generalization} with a recap of the results and directions for future research.

\section{Basic Notation and Poisson-based Model}\label{sect:notation-model}

We model the $G_t/G_t/\infty$ queue using a Poisson point process imbedded in a random environment. Let $(\Omega, \mathcal F, \PP)$ be a probability space with respect to which we define all random elements to follow.
The expectation with respect to $\PP$ is denoted as $\mathbb E(\cdot)$. 

\subsection{Model for the random environment.}\label{subsect:model-random-environment}

As mentioned in the introduction, our main contribution is to incorporate some fast 
oscillations modeled by an ergodic process $Z$ into the arrival rate of our queueing 
system. In this section we proceed to describe such a process.

\begin{hypothesis}\label{hyp:ergodicity}
Let $Z = \{Z_t \, ; \, t \geq 0\}$ be a $\R^d$-valued stochastic process defined on the probability space $(\Omega, \mathcal{F}, \PP)$. The initial distribution $\mathcal{L} (Z_0)$ of $Z$ is denoted by 
$\rho_0$, and we suppose that $Z$ possesses a unique invariant probability measure $\pi$. We also assume that $Z$ is strongly ergodic with rate $\ka > 0$ in the following sense: for any regular enough function $\psi : \R^d \to \R$, there exists a finite random variable $C = C_{\psi} (\omega) > 0$ such that $\PP$-almost surely we have
\begin{equation}\label{eq:hyp-ergodic-on-Z}
\left|\frac{1}{t} \iot \psi (Z_u) \, \dd u - \bar{\psi} \right|
\leq \frac{C}{(1+t)^{\kappa}} \, \, , \quad \text{with} \quad \bar{\psi} = \int_{\R^d} \psi(z)\,\pi(\dd z) \,.
\end{equation}
\end{hypothesis}

\begin{remark} {Note that the above hypothesis gives a 
convergence rate in the law of large number type statement. The constant
$C$ in general can depend on the realization of the random process $Z_\cdot$.}
There exists an abundant literature about results of the form \eqref{eq:hyp-ergodic-on-Z} for Markov chains. A general framework is developed in ~\cite{KLO}, 
which yields the following particular case: Set $Z_u = X_{[u]}$, where $[u]$ denotes the integer part of $u$ and $\{X_j \,;\, j \geq 0\}$ is a reversible ergodic Markov chain on a countable state space $E$. Let $\psi: E \to \R$ be such that $\sigma^2 (\psi) < \infty$, where
\begin{equation*}
\sigma^2 (\psi) = \lim_{N \to \infty} \, \frac{1}{N} \, \mathrm{Var} \left(\sum_{j=0}^{N-1} \psi (X_j) \right) .
\end{equation*}
{Then under some additional moment condition, it can be 
established that for any $0 < \ka < \frac{1}{2}$}, the following
holds almost surely:
\begin{equation}\label{eq:as-cvgce-olla}
\lim_{t \to \infty} \, (1+t)^{\ka} \left[ \frac{1}{t} \iot \psi (Z_u) \dd u - \bar{\psi}\right] = 0 \,,
\end{equation}
which immediately implies relation \eqref{eq:hyp-ergodic-on-Z}. 
Other examples of Markov processes (more specifically Harris chains) satisfying \eqref{eq:hyp-ergodic-on-Z} are provided in ~\cite{ChenLIL,LL}, based on the law of the iterated logarithm. Notice that \cite{LL} handles directly some continuous time Markov processes.
\end{remark}

\begin{remark}
\cref{hyp:ergodicity} can also be fulfilled in non-Markovian contexts. Indeed, consider an $\R^d$-valued fractional Brownian motion $B$ with Hurst parameter $H \in (0,1)$. Let $b: \R^d \to \R^d$ be a function such that the following inward property is satisfied for a constant $a>0$:
\begin{equation*}
\left\langle b(x) - b(y), x-y \right\rangle_{\R^d} \, \leq - a \left\Vert x-y \right\Vert^2 .
\end{equation*}
We consider the process $Z$ which solves to the following stochastic differential equation:
\begin{equation*}
Z_t = a + \iot b(Z_s) \, \dd s + B_t \, ,
\end{equation*}
where $a \in \R^d$. Then combining ~\cite{Sau} and ~\cite{GKN}, one can prove that $Z$ satisfies \cref{hyp:ergodicity} (details are omitted since this result is unrelated to the main message of the current paper). Observe that the case of a Brownian motion $B$ with $b=-a \, \textrm{Id}\, _{\R^d}$ corresponds to the classical Ornstein-Uhlenbeck case, for which \cref{hyp:ergodicity} thus holds. 
\end{remark}

\subsection{Model for the system state.}

Having specified our random environment, we now describe our model for the system state. It is determined by the arrival and service times that we proceed to define below.

Our random environment will enter into the intensity of the arrival process. Namely, the arrival process is conceived as follows.

\begin{hypothesis}\label{hyp:arr-times-Gamma}
Given $\ep > 0$, the sequence of arrival times $\{\Gamma_k^{\ep} \, ; k \geq 1 \}$ is distributed as a Poisson process with nonhomogeneous intensity $\{ \lambda(s) \psi(Z_{s/ \ep}) \, ; s \geq 0\}$, where $Z$ fulfills \cref{hyp:ergodicity} and $\psi : \R^d \to \R_+$ is a positive Lipschitz function. In addition, the functions $\la : \mathbb{R}_+ \to \mathbb{R}_+$ and $\psi : \R^d \to \mathbb{R}_+ $ are assumed to be continuous and bounded by a constant, say, $K > 0$.
\end{hypothesis}

We now turn our attention to the service process. It is given as a family $\{ L_k^{\ep} \, ; k \geq 1\}$ of random variables which are independent conditional on the arrival process $\{\Gamma_k^{\ep} \, ; k \geq 1 \}$.
Their law satisfies 
$\mathcal{L} (L_k^{\ep} | \Gamma_k^{\ep}) \equiv \nu (\Gamma_k^{\ep}, \dd r)$
for all $k$. Furthermore, we assume that the complementary cumulative distribution function
of the service times
\begin{equation}\label{eq:def-bar-F}
\bar F_s (r) := \int_r^\infty \nu(s,\dd \tau),
\quad r \geq 0
\end{equation}
satisfies the following tail and increment bounds.

\begin{hypothesis}\label{hyp:F-bar-increments}
There exists $\al > 0$ and a constant $c > 0$ such that
	\begin{eqnarray}\label{def:F-bar-generalization}
		\bar F_s (r) & \leq & c \left( \frac{1}{r^\alpha} \wedge 1 \right) \, , \quad \text{for all} \quad r, s > 0 \\
\text{and}\,\,\,
		\left|\bar F_s (r) - \bar F_t (r)\right| & \leq & c (1+r)^{-1-\al} (t-s) \, , \quad \text{for all} \quad 0 < s < t \quad \text{and} \quad r >0 . \notag
	\end{eqnarray}

\end{hypothesis}

\begin{remark}
\cref{hyp:F-bar-increments} covers a large variety of service time distributions, 
including both ``light tailed'' models such as the Gamma distribution, 
as well as ``heavy tailed'' models such as the Pareto-like
distributions. The latter is of particular interest, as attested by Resnick and Rootz{\'e}n \cite{RR}, for example.
Therefore, a typical example the reader might have in mind is given by
\begin{equation}\label{eq:example-for-F-bar}
\bar{F}_s (r)
= k_s (r) \, \mathbbm{1}_{[0,d_s]} (r) + \frac{c_s}{r^{\al}} \, \mathbbm{1}_{(d_s, \infty)} (r) \, ,
\end{equation}
where for each $s>0$, $k_s$ is a smooth function, where $d_s$ and $c_s$ are positive constants,  $\al >0$, and where proper assumptions are made so that $\bar{F}_s (r)$ is a continuous function.
Notice that a positive random variable $X$ whose distribution function $F$ satisfies~\eqref{eq:example-for-F-bar} is such that $E [X^{\beta}] < \infty$ for $\beta < \al$ and $E [X^{\beta}] = \infty$ for $\beta \geq \al$. Hence relation \eqref{eq:example-for-F-bar} allows a good calibration of the boundedness of moments for the service time.

\end{remark}

\begin{remark}
The function $s \mapsto c_s$ in \eqref{eq:example-for-F-bar} is thought of as a smooth and bounded slowly-varying function which modulates the service according to the arrival time. A specific example is given by the following oscillating function:
\begin{equation}\label{eq:def-c_s}
c_s = 1+ \beta \sin (ks) , \hspace{5pt} \text{with} \hspace{5pt} \beta \in (0,1) \hspace{5pt} \text{and} \hspace{5pt} k \geq 0.
\end{equation}
\end{remark}

With the arrival and service times in hand, our queueing system is classically described by a point process. Namely for $\ep > 0$, we consider the following counting measure on $\R_+ \times \R_+$:
\begin{equation}\label{eq:M-sum}
M^{\ep} := \sum_{k=1}^{\infty} \delta_{(\Gamma_k^{\ep}, L_k^{\ep})} \,.
\end{equation}
Then our main variable of interest is the number of active jobs in 
the infinite server queue at time $t$ which can be expressed as:
\begin{equation}\label{eq:N-t}
N^{\ep} (t) = \sum_{k=1}^{\infty} \mathbbm{1}_{\{\Gamma_k^{\ep} < t< \Gamma_k^{\ep} + L_k^{\ep} \}} = M^{\ep} \{(x,y) \in \mathbb{R}_+ \times \mathbb{R}_+, x< t< x+y\} .
\end{equation}
Our main aim in this paper is to derive a limit theorem for the process $N^{\ep} = \{N^{\ep} (t) ; t \in [0,T]\}$ as $\ep \to 0$, for an arbitrary time horizon $T>0$.

\begin{remark}\label{notation:quenched-annealed}
As the reader might have seen, there are two levels of randomness in our model. The first level corresponds to the random environment $Z$ described in \cref{subsect:model-random-environment}, while the second source of randomness is embodied in the Poisson point process $N^{\ep}$ 
given by \eqref{eq:N-t}. As in most of the literature on random environments, we shall play with the notion of annealed and quenched probabilities. The annealed probability represents the global probability with respect to all the randomness involved in our system and is denoted by $\PP$. The quenched probability corresponds to conditioning on the process $Z$, and is usually thought of as the physically observed probability (as already mentioned in the Introduction). This probability will be denoted by $\PP_Z$, with a corresponding expectation $\E_Z$. The relation between quenched and annealed probabilities is summarized as
\begin{equation}
\PP_Z (\cdot) = \PP (\cdot | Z) \,, \quad \text{and} \quad \E_Z [\cdot] = \E [\cdot | Z] \,.
\end{equation}
\end{remark}


\section{Analysis of the mean measure}\label{sect:law-N-ep}

In the previous section we have defined the Poisson point process $N^{\ep}$ describing our queueing system in a random environment. We now turn to the analysis of the mean measure $m^{\ep}$ of $N^{\ep}$, with a special focus on the asymptotics of $m^{\ep}$ as $\ep \to 0$.

\subsection{Mean measure of $N^\ep$}

This section is devoted to a full description of the law of $N^{\ep}$. The main result in this direction is summarized in the following proposition giving the conditional law of $N^{\ep} (t)$.

\begin{proposition}\label{prop:m-t-of-N-t}
Let $M^\ep$ and $\{N^{\ep} (t): t \geq 0\}$ be defined
by \eqref{eq:M-sum} and \eqref{eq:N-t}, respectively. Then under the quenched probability 
$\mathbb{P}_Z$,
$M^\ep$ is a Poisson random measure with mean measure given by
\begin{equation}\label{eq:def-mu-la-psi}
\tilde{\nu}^{\ep} (\dd x, \dd y) = \nu (x, \dd y) \mu^{\ep} (\dd x) \,,
\quad \text{with} \quad
\mu^{\ep} (\dd s) = \la (s) \psi (Z_{s/\ep}) \dd s \,,
\end{equation}
where we recall that $\nu$ is introduced in \eqref{eq:def-bar-F}. Furthermore,  we have that for any $t > 0$, 
$N^{\ep} (t)$ is a Poisson random variable with parameter
\begin{equation}\label{eq:m-t-def}
m^{\ep} (t) = \int_{\{(x,y): x<t< x+y\}} \nu (x, \dd y) \mu^{\ep} (\dd x) = \iot \int_{t-x}^{\infty} \nu (x, \dd y) \mu^{\ep} (\dd x).
\end{equation}
\end{proposition}

\begin{proof}
First, we will show that $M^\ep$ is a Poisson random measure. To this aim, recall \cref{notation:quenched-annealed} for the definition of the quenched probability $\PP_Z$. Then under $\PP_Z$ and according to \eqref{eq:M-sum}, the point process $M^{\ep}$ is of the form $\sum_{k \geq 1} \delta_{(\Gamma_k^{\ep} \,, \, L_k^{\ep})}$, where $\{\Gamma_k^{\ep}\, ; \, k \geq 1\}$ is a Poisson process (see \cref{hyp:arr-times-Gamma}). Thanks to \cite[Proposition 2.2]{Re},  we get that $M^{\ep}$ is a Poisson point process under $\PP_Z$, whose mean measure $\tilde{\nu}^{\ep}$ can be decomposed as
\begin{equation*}
\tilde{\nu}^{\ep} (\dd x, \dd y) = \nu (x, \dd y) \mu^{\ep} (\dd x) ,
\end{equation*}
where $\nu$ is the measure featuring in \eqref{eq:def-bar-F} and $\mu^{\ep}$ is defined by \eqref{eq:def-mu-la-psi}.

Therefore, according to \cite[Chapter VI Theorem 2.9]{Cin}, the quenched Laplace transform of $M^{\ep}$ is given for all measurable and positive functions $f: \mathbb{R}_+ \times \mathbb{R}_+ \rightarrow \mathbb{R}$ by
\begin{equation}\label{eq:LT-Mf}
\mathbb{E}_{Z}\lc   e^{-M^{\ep} f} \rc 
= e^{- \tilde{\nu}^{\ep} (1-e^{-f})} .
\end{equation}
where the notation $\tilde{\nu}^{\ep} (1-e^{-f})$ in the above stands for the integral of $(1-e^{-f})$ with respect to the measure $\tilde{\nu}^{\ep} = \nu (x, \dd y) \mu^{\ep} (\dd x)$. Hence, we can rewrite \eqref{eq:LT-Mf} as
\begin{multline}\label{eq:exp-e-Mf}
\mathbb{E}_{Z} \lc 
\exp \Big\{ - \int_{\mathbb{R}_+ \times \mathbb{R}_+} M^{\ep} (\dd x, \dd y) f(x, y) \Big\} \rc \\
= \exp \Big\{- \int_{\mathbb{R}_+ \times \mathbb{R}_+} \nu (x, \dd y) \mu^{\ep} (\dd x) (1-e^{-f(x, y)}) \Big\} .
\end{multline}

\noindent
In addition, according to \eqref{eq:N-t} we have $N^{\ep} (t) = M^{\ep} f$ with $f(x,y) = \mathbbm{1}_{(x<t<x+y)}$. Plugging this expression into \eqref{eq:exp-e-Mf}, 
we immediately get relation \eqref{eq:m-t-def} completing the proof.
\end{proof}

\subsection{Limit for $m^\ep(t)$}
According to relation \eqref{eq:m-t-def} in \cref{prop:m-t-of-N-t} and the fact that $\mu^{\ep} (\dd s) = \la (s) \psi (Z_{s/ \ep}) \dd s$, the (quenched) mean of the random variable $N^{\ep} (t)$ defined in~\eqref{eq:N-t} is given by
\begin{equation}\label{eq:def-m-t-epsilon}
m^{\ep}(t) = \int_{0}^{t} \left(\lambda (s) \psi (Z_{s/\ep}) \int_{t-s}^{\infty} \nu (s, \dd r)\right) \dd s
=
\iot \, \la(s) \, \psi(Z_{s/\ep}) \, \bar{F}_s (t-s) \, \dd s ,
\end{equation}
where $\bar{F}_s$ is the tail function given by \eqref{eq:def-bar-F}. The following theorem gives a full description of the almost sure asymptotic behavior of $m^{\ep} (t)$.

\begin{theorem}\label{Thm:limit-m-ep}
Let us assume the same set-up and notations as in 
\cref{prop:m-t-of-N-t}. Furthermore,
suppose that \cref{hyp:ergodicity,hyp:arr-times-Gamma,hyp:F-bar-increments} are satisfied.
Then under the quenched probability $\PP_Z$, for any $t>0$ we have almost surely
\begin{equation}\label{eq:lim-m-ep}
\lim_{\ep \rightarrow 0} m^{\ep} (t) = \bar{m} (t) \,,
\end{equation}
where $\bar{m} (t)$ is given by
\begin{equation}\label{eq:m-bar}
\bar{m} (t) = \sigma(t) \bar{\psi} \, ,
\end{equation}
and for which the quantities $\sigma (t)$ and $\bar{\psi}$ are respectively defined by
\begin{equation}\label{eq:quantity-sigma-psi-bar}
\sigma(t) = \iot \la (s) \bar{F}_s (t-s) \dd s \,,
\quad \text{and} \quad
\bar\psi = \int_{\R^d} \psi(z)\,\pi(\dd z) \,.
\end{equation}
In relation \eqref{eq:quantity-sigma-psi-bar}, $\pi$ is the invariant
measure of the process $Z$ introduced in \cref{hyp:ergodicity}.
\end{theorem}

\begin{proof}
Starting from \eqref{eq:def-m-t-epsilon} and upon introducing the notation $h(s,t) = \lambda(s)\bar{F}_s(t-s)$, we write $m^{\ep} (t)$ as
\begin{equation}\label{eq:m-ep-with-h_s-t}
m^\ep(t) = \int_0^t h(s,t)\psi(Z_{s/\ep})\,\dd s.
\end{equation}
We then compute, making use of a simple integration by parts,
\begin{eqnarray}
m^\ep(t) & = & 
\int_0^t h(s,t)\frac{\dd}{\dd s}\left(\int_0^s\psi(Z_{r/\ep})\,\dd r\right)\,\dd s \notag \\
& = & 
\left.h(s,t)\int_0^s\psi(Z_{r/\ep})\,\dd r\right|_{s=0}^{s=t}
-
\int_0^t \frac{\dd}{\dd s}h(s,t)\left(\int_0^s\psi(Z_{r/\ep})\,\dd r\right)\,\dd s \notag \\
& = & 
A^{\ep}_1 (t) - A^{\ep}_2 (t) \,, \label{eq:m-ep-with-h(s,t)}
\end{eqnarray}
where we have set
\begin{equation}\label{eq:A-ep-terms}
A^{\ep}_1 (t) = h(t,t)\int_0^t\psi(Z_{r/\ep})\,\dd r \,, \quad \text{and} \quad
A^{\ep}_2 (t) = \int_0^t \frac{\dd}{\dd s}h(s,t)\left(\int_0^s\psi(Z_{r/\ep})\,\dd r\right)\,\dd s \,.
\end{equation}
We now treat the limits of $A^{\ep}_1 (t)$ and $A^{\ep}_2 (t)$ separately.

The term $A^{\ep}_1 (t)$ can be analyzed as follows. The elementary change of variables $r := r/\ep$ yields
\begin{equation*}
A^{\ep}_1 (t) = h(t,t)t \,\lim_{\ep\to 0} \, \frac{\ep}{t}\int_0^{t/\ep} \psi(Z_{r})\,\dd r\,.
\end{equation*}
Hence, invoking \cref{hyp:ergodicity} we get
\begin{equation}\label{eq:A1andA2-limit}
\lim_{\ep\to 0} A^{\ep}_1 (t)
= h(t,t)t\bar{\psi} \,, \qquad \PP_Z-\text{a.s.}
\end{equation}

For $A^{\ep}_2 (t)$, we add and subtract $\bar{\psi}$ to get
\begin{eqnarray}
A^{\ep}_2 (t)
& = & 
\int_0^t \left(\frac{\dd}{\dd s}h(s,t)\right) s
\left[
\frac{1}{s}\int_0^s\psi(Z_{r/\ep})\, \dd r - \bar{\psi} + \bar{\psi}
\right]\,\dd s \notag \\
& = & 
\int_0^t \left(\frac{\dd}{\dd s}h(s,t)\right)s\bar{\psi}\,\dd s
+
\int_0^t \left(\frac{\dd}{\dd s}h(s,t)\right)s
\left[
\frac{1}{s}\int_0^s\psi(Z_{r/\ep})\,\dd r - \bar{\psi}
\right]\,\dd s \notag\\
& \equiv &
A^{\ep}_{2,1} (t) + A^{\ep}_{2,2} (t) \,. \label{eq:A3-secondintegral}
\end{eqnarray}

\noindent
We now proceed to bound the term $A^{\ep}_{2,2} (t)$ in relation \eqref{eq:A3-secondintegral}. Namely, a trivial integral bound  and the same change of variables as for $A^{\ep}_1 (t)$ enable us to write
\begin{equation*}
\left| A^{\ep}_{2,2} (t) \right|
\leq
\iot s \left| \frac{\dd}{\dd s}h(s,t)\right|
\left|
\frac{\ep}{s}\int_0^{s/\ep}\psi(Z_{r})\,\dd r - \bar{\psi}
\right|\,\dd s \,.
\end{equation*}
Hence, invoking \cref{hyp:ergodicity} we get
\begin{equation*}
\left| A^{\ep}_{2,2} (t) \right|
\leq
C \iot s \left| \frac{\dd}{\dd s}h(s,t) \right|
(1 + s/\ep)^{-\ka}
\,\dd s \,,
\end{equation*}
and applying the Dominated Convergence Theorem we end up with
\begin{equation}\label{eq:limit-A-ep-2-2}
\lim_{\ep \to 0} A^{\ep}_{2,2} (t) = 0 \,.
\end{equation}
We also notice that the term $A^{\ep}_{2,1} (t)$ introduced in \eqref{eq:A3-secondintegral} can be simplified thanks to an elementary integration by parts. We get
\begin{equation}\label{eq:m-ep-limit1-int}
A^{\ep}_{2,1} (t) =
\int_0^t \left(\frac{\dd}{\dd s}h(s,t)\right)s\bar{\psi}\,\dd s
=
h(t,t) t \bar{\psi} - \bar{\psi} \iot h(s,t)\,\dd s \,,
\end{equation}
the right-hand side of which is finite, owing to \cref{hyp:arr-times-Gamma,hyp:F-bar-increments}. Therefore plugging~\eqref{eq:m-ep-limit1-int} and \eqref{eq:limit-A-ep-2-2} into relation \eqref{eq:A3-secondintegral}, we have obtained
\begin{equation}\label{eq:A2(t)-ep-limit}
\lim_{\ep \to 0} A^{\ep}_2 (t) = h(t,t) t \bar{\psi} - \bar{\psi} \iot h(s,t)\,\dd s \,.
\end{equation}

We can now conclude as follows. Gathering \eqref{eq:A1andA2-limit} and \eqref{eq:A2(t)-ep-limit} into \eqref{eq:m-ep-with-h(s,t)}, we have
\begin{equation*}
\lim_{\ep \to 0} m^{\ep} (t) =
\bar{\psi} \iot h(s,t)\,\dd s \,, \quad \PP_Z-\text{a.s.} \, ,
\end{equation*}
which is exactly our claim \eqref{eq:lim-m-ep}. This completes the proof.
\end{proof}

\section{Homogenized Process}\label{sect:homogenized-process}

With the limiting behavior of $m^{\ep}$ in hand, we are now ready to give the asymptotic description of the process $N^{\ep}$. As usual, we will decompose this analysis into a study of the convergence of the 
finite dimensional distributions and a tightness result.

\subsection{Limit for the finite dimensional distributions}

In order to simplify our presentation, we will first derive the limit of bivariate quantities of the form $(N^{\ep} (t_1), N^{\ep} (t_2) )$ for two instants $t_1 < t_2$. 
To this aim, inside the quadrant $\R_+ \times \R_+$, we will consider
three disjoint regions $\{A_i \,, \, i = 1,2,3\}$ defined as follows (see \cref{fig:A_i-region}):
\begin{eqnarray}
A_1 &=& \{(\ga, l) \in \mathbb{R}_+ \times \mathbb{R}_+ : \ga \leq t_1  \text{ and }  t_1 < \ga+l \leq t_2\} \,;
\label{def:region_A1}\\
A_2 &=& \{(\ga, l)\in \mathbb{R}_+ \times \mathbb{R}_+ : \ga \leq t_1   \text{ and }  t_2 < \ga + l\} \,;
\label{def:region_A2}\\
A_3 &=& \{(\ga, l) \in \mathbb{R}_+ \times \mathbb{R}_+ : t_1 < \ga \leq t_2   \text{ and }  t_2 < \ga+l\} \,.
\label{def:region_A3}
\end{eqnarray}

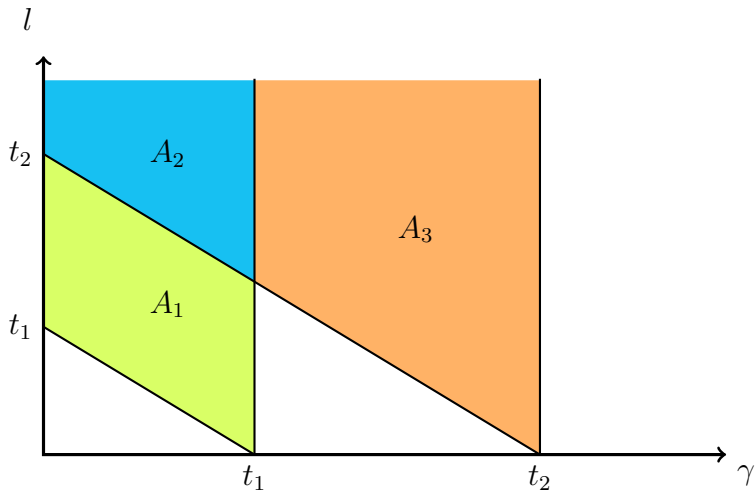
\begin{figure}
\begin{tikzpicture}[xscale = 1.65]
\path [fill = lime!60] (1.7,0) -- (1.7,2.3) -- (0,4) -- (0,1.7);
\path [fill = cyan!70] (1.7,2.3) -- (1.7,4.98) -- (0,4.98) -- (0,4);
\path [fill = orange!60] (4,0) -- (4,4.98) -- (1.7,4.98) -- (1.7,2.3);

\draw [very thick] [<->] (0,5.3) -- (0,0) -- (5.5,0);
\draw [thick] (0,1.7) -- (1.7,0);
\node [left] at (0,1.7) {$t_1$};
\draw [thick] (0, 4) -- (4 ,0);
\node [left] at (0,4) {$t_2$};
\draw [thick] (1.7,0) -- (1.7, 5);
\node [below] at (1.7,0) {$t_1$};
\draw [thick] (4, 0) -- (4,5);
\node [below] at (4,0) {$t_2$};
\node [above left] at (0,5.5) {$l$};
\node [below right] at (5.5,0) {$\gamma$};

\node at (1, 2) {$A_1$};
\node at (1, 4) {$A_2$};
\node at (3, 3) {$A_3$};
\end{tikzpicture}
\caption{Three disjoint regions used for the limit of bivariate quantities.}
\label{fig:A_i-region}
\end{figure}

\noindent
Notice that since the $A_i$'s are are disjoint, the quantities $\{M^{\ep} (A_i); \, i = 1,2,3\}$ are independent Poisson random variables. 
Similar to the proof of \eqref{eq:m-t-def}, 
their respective quenched means are given by (see \cref{prop:m-t-of-N-t})
\begin{equation}\label{eq:def-mean-M-ep}
m^{\ep}_i := \mathbb{E}_Z [M^{\ep} (A_i)] = \int_{A_i} \nu (s, \dd r) \la (s) \psi (Z_{s/\ep}) \dd s \,, \quad \text{for}\,\,\,i=1,2,3.
\end{equation}

\noindent
Now observe that the vector $(N^{\ep} (t_1) , N^{\ep} (t_2))$ can be decomposed as
\begin{equation}\label{eq:N-ep(t)}
N^{\ep} (t_1) = M^{\ep} (A_1) + M^{\ep} (A_2) , \hspace{20pt} N^{\ep} (t_2) = M^{\ep} (A_2) + M^{\ep} (A_3) .
\end{equation}
As a consequence, the means $m^{\ep} (t_1)$, $m^{\ep} (t_2)$ can also be written in terms of the $m_{i}^{\ep}$'s given by \eqref{eq:def-mean-M-ep}:
\begin{equation}\label{eq:m-ep(t)}
m^{\ep} (t_1) = m^{\ep}_1 + m^{\ep}_2 , \hspace{20pt} m^{\ep} (t_2) = m^{\ep}_2 + m^{\ep}_3 .
\end{equation}

We now state a proposition giving the quenched limit in law for $(N^{\ep} (t_1) , N^{\ep} (t_2))$. 

\begin{proposition}\label{Thm:convergence-N-ep(t1t2)}
Let $M^{\ep}$ be the Poisson random measure on $\mathbb{R}_+ \times \mathbb{R}_+$ defined by \eqref{eq:M-sum}, with mean measure $\tilde{\nu}^{\ep} (\dd s, \dd r) = \nu (s,\dd r) \mu^{\ep} (\dd s)$, as given in \eqref{eq:def-mu-la-psi}.
Assume that the conditions in \cref{hyp:ergodicity,hyp:arr-times-Gamma,hyp:F-bar-increments} are met.
Then for any fixed two time points $0 \leq t_1 < t_2$,
we have the following statements.

\begin{enumerate}[wide, labelwidth=!, labelindent=0pt, label=\emph{(\roman*)}]
\setlength\itemsep{.1in}

\item\label{it:convergence-nep-i}
$\mathbb{P}_Z$-almost surely we have that for all $\xi_1$, $\xi_2>0$,
\begin{equation}\label{eq:lim-E_Z-N-twotp}
\lim_{\ep \to 0} \mathbb{E}_Z [e^{- (\xi_1 N^{\ep} (t_1) + \xi_2 N^{\ep} (t_2))}]
=
\mathbb{E} [e^{-(\xi_1 N(t_1) + \xi_2 N(t_2))}] \,.
\end{equation}
In the right-hand side of equation \eqref{eq:lim-E_Z-N-twotp}, the process $\{N (t) ; \, t \geq 0\}$is independent of $N^{\ep}$ and is defined similarly to that of \eqref{eq:N-t}, albeit in a non-random environment. Namely, $N$ can be expressed as
\begin{equation}\label{def-eq:N-t}
N(t) = \sum_{k=1}^{\infty} \mathbbm{1}_{\{ \Gamma_k < t < \Gamma_k + L_k \}}
=
M \{ (x, y) \in \mathbb{R}_+ \times \mathbb{R}_+ , x <t< x+y\} \,,
\end{equation}
with $M$ being a Poisson point process on $\mathbb{R}_+ \times \mathbb{R}_+$ of the form $M = \sum_{k=1}^{\infty} \delta_{(\Gamma_k, L_k)}$. The mean measure of $M$ is given by
\begin{equation}\label{eqdef:nu-meanmeasure}
\tilde{\nu} (\dd s, \dd r) = \bar{\psi} \la (s) \nu (s, \dd r) \dd s \,,
\end{equation}
where we recall that $\bar{\psi}$ is defined by \eqref{eq:quantity-sigma-psi-bar}.

\item\label{it:convergence-nep-ii}
$\mathbb{P}_Z$-almost surely we have the following limit in law as $\ep \to 0$:
\begin{equation*}
(N^{\ep} (t_1), N^{\ep} (t_2)) \xrightarrow{(\dd)} (N(t_1), N(t_2)) .
\end{equation*}
\end{enumerate}
\end{proposition}

\begin{remark}
In order to alleviate notations, we have assumed that our underlying probability space carries the family $\{ N^{\ep} ;\, \ep \geq 0\}$ as well as the process $N$ defined by \eqref{def-eq:N-t}. This explains why we have expressed \eqref{eq:lim-E_Z-N-twotp} with the same expectation $\E$ on both sides of the relation.
\end{remark}

\begin{proof}[Proof of \cref{Thm:convergence-N-ep(t1t2)}]
We first introduce the notation that will be used in the proof. For $i = 1,2,3,$ and $\la_i > 0$, we define the functions $f_i := \la_i \mathbbm{1}_{A_i} $, 
where the sets $A_i$'s are given by \eqref{def:region_A1}-\eqref{def:region_A3}. We note that $M^{\ep} (f_i) = \la_i M^{\ep} (A_i)$. By considering the Laplace 
transform of $M^{\ep} (A_i)$, we can see that Theorem 2.9 in Chapter VI of \cite{Cin} yields
\begin{equation}\label{eq:LaplaceTransform-M-ep}
\mathbb{E}_Z [e^{-\sum_{i = 1}^{3} \la_i M^{\ep} (A_i)}]
=
\mathbb{E}_Z [e^{-\sum_{i=1}^{3} M^{\ep} (f_i)}]
=
e^{-\sum_{i=1}^{3} \tilde{\nu}^{\ep} (1-e^{-f_i})} \,,
\end{equation}
where $\tilde{\nu}^{\ep}$ is the measure defined by \eqref{eq:def-mu-la-psi}.
In equation \eqref{eq:LaplaceTransform-M-ep}, we specify again that $\tilde{\nu}^{\ep} (1-e^{-f_i})$ stands for the integral of $(1-e^{-f_i})$ with respect to the measure $\tilde{\nu}^{\ep}$, as with \eqref{eq:LT-Mf}. We now split the analysis of \eqref{eq:LaplaceTransform-M-ep} into several steps.

\noindent
{\it Step 1: Decomposition of the Laplace transform.} Taking into account the expression \eqref{eq:def-mu-la-psi} for $\tilde{\nu}^{\ep}$, the right-hand side of \eqref{eq:LaplaceTransform-M-ep} can be rewritten as
\begin{equation}\label{eq:LaplaceTransform-M-ep-ver1}
e^{-\sum_{i=1}^{3} \tilde{\nu}^{\ep} (1-e^{-f_i})}
=
\prod_{i=1}^{3} \exp \left\{-\tilde{\nu}^{\ep} (1-e^{-f_i}) \right\}
=
\prod_{i=1}^{3} \exp \left\{ - G_i^{\ep} \right\} ,
\end{equation}
where each function $G_i^{\ep}$, for $i = 1,2,3,$ is given as the following integral,
\begin{eqnarray}
G_i^{\ep} & = & \int_{\R_+ \times \R_+} (1-e^{-f_i}) \nu (s, \dd r) \mu^{\ep} (\dd s) \notag \\
& = & \int_{\R_+ \times \R_+} (1-e^{-f_i}) \nu (s, \dd r) \la (s) \psi (Z_{s/ \ep}) \dd s \,. \label{eq:G_i_ep}
\end{eqnarray}
In the sequel, we shall characterize the limit of each $G_i^{\ep}$.
Note first that owing to the relation $f_i = \la_i \mathbbm{1}_{A_i}\,$, we have
\begin{equation*}
1 - e^{- f_i (s,r)} = (1 - e^{- \la_i} ) \mathbbm{1}_{A_i} (s,r) .
\end{equation*}
Therefore, one can recast the term $G^{\ep}_i$ as
\begin{equation}\label{eqdef:m1-ep}
G_i^{\ep} = (1 - e^{- \la_i}) m_i^{\ep} \,,
\quad \text{where we recall that}\,\,\,
m_i^{\ep} = \int_{A_i} \nu (s, \dd r) \la (s) \psi (Z_{s/\ep}) \dd s .
\end{equation}

\noindent
In summary, substituting \eqref{eqdef:m1-ep} into \eqref{eq:LaplaceTransform-M-ep-ver1} and then \eqref{eq:LaplaceTransform-M-ep}, we have shown that
\begin{equation}\label{eq:LaplaceTransform-M-ep-ver2}
\E_Z [e^{-\sum_{i = 1}^{3} \la_i M^{\ep} (A_i)}] =
\prod_{i=1}^{3} \exp \left\{ -  (1 - e^{- \la_i}) m_i^{\ep} \right\}.
\end{equation}
We are now reduced to an examination of $\lim_{\ep\to 0} m^{\ep}_i$ in the right-hand side of equation~\eqref{eq:LaplaceTransform-M-ep-ver2}.

\noindent
{\it Step 2: Analysis of an integral with fast oscillatory integrand.}
In order to handle the terms $m^{\ep}_i$ in \eqref{eqdef:m1-ep}, we will generalize slightly the analysis of integral expressions like \eqref{eq:m-ep-with-h_s-t}. Namely, consider a continuously differentiable function $g: \R_+ \to \R$ and for $\ep >0$,  let $I^\ep(\tau_1, \tau_2)$ be given in the following integral:
\begin{equation}\label{eq:I-ep-of-tau}
I^\ep(\tau_1, \tau_2) 
= \int_0^{\tau_1} g(s,\tau_2)\psi \left(Z_{s/\ep}\right)\,\dd s \,,
\end{equation}
for $0 \leq \tau_1 \leq \tau_2$. Then following the same integration by parts procedure as for \eqref{eq:m-ep-with-h(s,t)} in the proof of \cref{Thm:limit-m-ep}, we have
\begin{eqnarray*}
I^\ep(\tau_1, \tau_2) 
& = & 
\int_0^{\tau_1} g(s,\tau_2)
\left(\frac{\dd}{\dd s}\int_0^s
\psi\left(Z_{r/\ep}\right)\,\dd r\right)\,\dd s\\
& = & \left.
g(s,\tau_2) 
\int_0^s\psi\left(Z_{r/\ep}\right)\,\dd r
\right|_{s=0}^{s=\tau_1}
-\int_0^{\tau_1} \frac{\dd}{\dd s}g(s,\tau_2)
\left(\int_0^s\psi\left(Z_{r/\ep}\right)\, \dd r\right)\,\dd s\\
& = & 
g(\tau_1,\tau_2) 
\int_0^{\tau_1}\psi\left(Z_{r/\ep}\right)\,\dd r
-\int_0^{\tau_1} \frac{\dd}{\dd s}g(s,\tau_2)
\left(\int_0^s\psi\left(Z_{r/\ep}\right)\,\dd r\right)\,\dd s \,.
\end{eqnarray*}
Then following the same steps as for the analysis of $A^{\ep}_1 (t)$ and $A^{\ep}_2 (t)$ in the proof of \cref{Thm:limit-m-ep} (see respectively \eqref{eq:A1andA2-limit} and \eqref{eq:A2(t)-ep-limit}),
we compute the limit of $I^\ep(\tau_1, \tau_2)$ as $\ep \to 0$:
\begin{eqnarray}
\lim_{\ep\to 0}I^\ep(\tau_1, \tau_2) 
& = & 
g(\tau_1,\tau_2) \tau_1\bar{\psi}
-\int_0^{\tau_1} \left(\frac{\dd}{\dd s}g(s,\tau_2)\right)s\bar{\psi}\,\dd s \notag\\
& = & 
\left[g(\tau_1, \tau_2)\tau_1  
- g(s,\tau_2)s\bigg|_{s=0}^{s=\tau_1} 
+ \int_0^{\tau_1}g(s,\tau_2)\,\dd s\right]\bar{\psi} \notag \\
& = & 
\left[\int_0^{\tau_1}g(s,\tau_2)\,\dd s\right]\bar{\psi}\,, \label{eq:limit-I-ep-for-tau}
\end{eqnarray}
where the second equality is obtained thanks to another use of the integration by parts formula.

\noindent
{\it Step 3: Analysis of $m_i^{\ep}$, for $i=1, 2, 3$.}
We will now resort to the general identity \eqref{eq:limit-I-ep-for-tau} in order to analyze the terms $m^{\ep}_i$ in \eqref{eqdef:m1-ep}. 
We start by recasting $m^{\ep}_1$ as an expression involving \eqref{eq:I-ep-of-tau}. Namely, invoking the definition \eqref{eq:def-bar-F} of $\bar{F}_s$, we write
\begin{align}
m_1^{\ep} 
& = \int_{0}^{t_1} \bigg( \la (s) \psi (Z_{s/\ep}) \int_{t_1 - s}^{t_2 - s} \nu (s, \dd r) \bigg) \dd s \notag\\ 
&= \int_{0}^{t_1}  \la (s) \psi (Z_{s/\ep}) \Big(\bar{F}_s (t_1-s) - \bar{F}_s (t_2 - s) \Big) \dd s \notag\\
& = \int_{0}^{t_1} \la (s)\bar{F}_s (t_1-s)\psi (Z_{s/\ep})\, \dd s 
- \int_{0}^{t_1} \la (s) \bar{F}_s (t_2-s)\psi (Z_{s/\ep}) \, \dd s \,. \label{eq:m1-ep-ver2}
\end{align}
Upon setting
\begin{equation}\label{eq:g-s-t-for-h_st}
g(s,t) = \la(s)\bar{F}_s(t-s),
\end{equation}
and recalling \eqref{eq:I-ep-of-tau}, we can rewrite the expression \eqref{eq:m1-ep-ver2} of $m^{\ep}_1$ as follows:
\begin{equation*}
m^{\ep}_1 = \int_{0}^{t_1} g(s,t_1) \psi (Z_{s/\ep}) \dd s - \int_{0}^{t_1} g(s, t_2) \psi (Z_{s/\ep}) \dd s
= I^{\ep} (t_1, t_1) - I^{\ep} (t_1, t_2) \,.
\end{equation*}

\noindent
Due to the fact that $\la$ and $\bar{F}_s$ are continuous and bounded functions, we can apply directly the result \eqref{eq:limit-I-ep-for-tau} from Step 2. We get the following $\PP_Z$-almost sure limit for $m^{\ep}_1$:
\begin{eqnarray}\label{eq:limit-m1-ep}
\lim_{\ep \to 0} m_{1}^{\ep} 
&=& \left[\int_0^{t_1}g(s,t_1)\,\dd s\right]\bar{\psi} - \left[\int_0^{t_1}g(s,t_2)\,\dd s\right]\bar{\psi} 
\notag \\
&=& \left[\int_{0}^{t_1} \lc g(s,t_1) - g(s, t_2) \rc \,\dd s \right] \bar{\psi}
\equiv m_1\,,
\end{eqnarray}
where we recall that the function $g$ is given by \eqref{eq:g-s-t-for-h_st}.

The analysis of $m^{\ep}_2$ and $m^{\ep}_3$ are obtained along similar lines.
Hence we will just write down the main steps and invite the patient reader to 
fill in the corresponding details. First we have the following expressions
\begin{eqnarray}
m_{2}^{\ep}
&=&
\int_{0}^{t_1} \la (s) \bar{F}_s (t_2-s)\psi (Z_{s/\ep})\dd s \,; \notag\\
m_{3}^{\ep}
&=&
 \int_{0}^{t_2} \la (s) \bar{F}_s (t_2 - s)\psi (Z_{s/\ep})\dd s 
- \int_{0}^{t_1} \la (s) \bar{F}_s (t_2 - s)\psi (Z_{s/\ep})\dd s \,. \label{eq:m3-ep-ver2}
\end{eqnarray}
Then we can obtain a $\PP_Z$-almost sure limit of the form
\begin{align}
\lim_{\ep \to 0} m_{2}^{\ep} 
& = \left[\int_0^{t_1}g(s,t_2)\,\dd s\right]\bar{\psi} \equiv m_2\,;\label{eq:limit-m2-ep}\\
\lim_{\ep \to 0} m_{3}^{\ep} 
&= \left[\int_0^{t_2}g(s,t_2)\,\dd s\right]\bar{\psi} - \left[\int_0^{t_1}g(s,t_2)\,\dd s\right]\bar{\psi}
= \left[\int_{t_1}^{t_2} g(s,t_2) \,\dd s\right] \bar{\psi} \equiv m_3\,. \label{eq:limit-m3-ep}\
\end{align}

\noindent
Hence, gathering \eqref{eq:limit-m1-ep}, \eqref{eq:limit-m2-ep}, and \eqref{eq:limit-m3-ep} into relation \eqref{eq:LaplaceTransform-M-ep-ver2}, we have obtained
\begin{equation}\label{eq:LaTrans-M-ep}
\lim_{\ep \to 0} \,\E_Z [e^{-\sum_{i = 1}^{3} \la_i M^{\ep} (A_i)}] =
\prod_{i=1}^{3} \exp \left\{ -  (1 - e^{- \la_i}) m_i \right\} .
\end{equation}

\noindent
{\it Step 4: Final concluding steps.}
Recall that our aim is to analyze the left-hand side of relation \eqref{eq:lim-E_Z-N-twotp}. To this aim, notice that owing to relation \eqref{eq:N-ep(t)}, we have
\begin{equation*}
\E_Z [e^{- (\xi_1 N^{\ep} (t_1) + \xi_2 N^{\ep} (t_2))}]
=
\E_Z [e^{-\sum_{i = 1}^{3} \la_i M^{\ep} (A_i)}] \,,
\end{equation*}
where we have set
\begin{equation}\label{eq:LaTrans-constants}
\la_1 = \xi_1 \,, \quad \la_2 = \xi_1+\xi_2 \,, \quad \text{and} \quad \la_3 = \xi_2 \,.
\end{equation}
Therefore, an immediate application of \eqref{eq:LaTrans-M-ep} yields
\begin{equation}\label{eq:LaTrans-M-Ver}
\lim_{\ep \to 0} \, \E_Z \lc e^{- (\xi_1 N^{\ep} (t_1) + \xi_2 N^{\ep} (t_2))} \rc
=
\prod_{i=1}^{3} \exp \left\{ -  (1 - e^{- \la_i}) m_i \right\} .
\end{equation}
By the definition $f_i := \la_i \mathbbm{1}_{A_i}$ given at the beginning of our proof and the expressions ~\eqref{eq:limit-m1-ep}, ~\eqref{eq:limit-m2-ep}, and ~\eqref{eq:limit-m3-ep} for $m_1$, $m_2$, and $m_3$, respectively, relation \eqref{eq:LaTrans-M-Ver} can be rewritten as
\begin{equation}\label{eq:LaTrans-M-Ver-expsum}
\lim_{\ep \to 0} \, \E_Z \lc e^{- (\xi_1 N^{\ep} (t_1) + \xi_2 N^{\ep} (t_2))} \rc
=
e^{-\sum_{i=1}^{3} \tilde{\nu} (1-e^{-f_i})} \,,
\end{equation}
where $\tilde{\nu}$ stands for the mean measure of the point process $M$ (see equation \eqref{eqdef:nu-meanmeasure}).
Taking into account the values \eqref{eq:LaTrans-constants} for $\la_1$, $\la_2$, $\la_3$ and the definition \eqref{def-eq:N-t} of the process $N$, then by Chapter VI Theorem 2.9 in ~\cite{Cin} we end up with
\begin{equation*}
e^{-\sum_{i=1}^{3} \tilde{\nu} (1-e^{-f_i})}
=
\E \lc e^{-\sum_{i=1}^{3} M (f_i)} \rc
=
\E \lc e^{-\sum_{i = 1}^{3} \la_i M (A_i)} \rc
=
\E \lc e^{-(\xi_1 N(t_1) + \xi_2 N(t_2))} \rc \,,
\end{equation*}
which leads to our claim \eqref{eq:lim-E_Z-N-twotp}. This proves statement \emph{\ref{it:convergence-nep-i}}.

Finally, 
in the light of statement \emph{\ref{it:convergence-nep-i}} and Laplace transform properties, 
statement \emph{\ref{it:convergence-nep-ii}} holds.
\end{proof}

Once the limit for the bivariate vector $(N^{\ep} (t_1), N^{\ep} (t_2))$ is obtained, the extension to the multivariate case can be done through routine (though tedious) considerations. We state this generalization and a sketch of its proof below.

\begin{corollary}\label{Cor:convergence-N-ep}
With the assumptions in \cref{Thm:convergence-N-ep(t1t2)}, let $t_1< t_2< \cdots <t_n$ be fixed. Then $\mathbb{P}_Z$-almost surely we have that
\begin{equation*}
(N^{\ep} (t_1), N^{\ep} (t_2), \ldots , N^{\ep} (t_n)) \xrightarrow{(\dd)} (N(t_1), N(t_2), \ldots , N(t_n)) \,.
\end{equation*}
\end{corollary}

\begin{proof}
Following the same procedures as in the proof of ~\cref{Thm:convergence-N-ep(t1t2)}, we first decompose the quadrant $\R_+ \times \R_+$ into several disjoint regions (see \cref{fig:A_i-region-n-timepoints} for a depiction of the sets $A_{i,j}$):
\[
\{A_{i, j} : \text{for} \,\, 1\leq i \leq n \,\, \text{and}\,\, 1 \leq j \leq n-i+1 \} \,,
\]
where for each $i$ and $j$, the region is defined as
\[
A_{i, j} = \left\{(\gamma, l) \in \R_+ \times \R_+: \, t_{i-1} \leq \gamma \leq t_i \,\,\text{and} \,\,t_j \leq \gamma + l \leq t_{j+1} \right\} ,
\]
with the additional convention $t_0 = 0$ and $t_{n+1} = \infty$. Since the $A_{i,j}$'s are disjoint regions, the quantities $\{M^{\ep} (A_{i,j}) : i, j = 1,2, ... , n\}$ are independent Poisson random variables. Similar to \eqref{eq:def-mean-M-ep}, their respective quenched means are given by
\begin{equation*}
\E_Z [M^{\ep} (A_{i,j})] = \int_{A_{i,j}} \nu (s, \dd r) \la (s) \psi (Z_{s/\ep}) \dd s \,.
\end{equation*}
Furthermore, as with \eqref{eq:N-ep(t)} and \eqref{eq:m-ep(t)}, the quantity $N^{\ep} (t_i)$ can be written as
\begin{equation*}
N^{\ep} (t_i) = \sum_{j = 1}^{n-i+1} M^{\ep} (A_{i,j}) \,,
\end{equation*}
whose mean can be expressed as
\begin{equation*}
\E_Z[N^{\ep} (t_i)] =
m^{\ep} (t_i) = \sum_{j = 1}^{n-i+1} \E_Z [M^{\ep} (A_{i,j})] \,.
\end{equation*}
Starting from this set of relations, the rest of the proof will be a repetition of 
that for \cref{Thm:convergence-N-ep(t1t2)}. We omit the details for the sake of conciseness.
\end{proof}

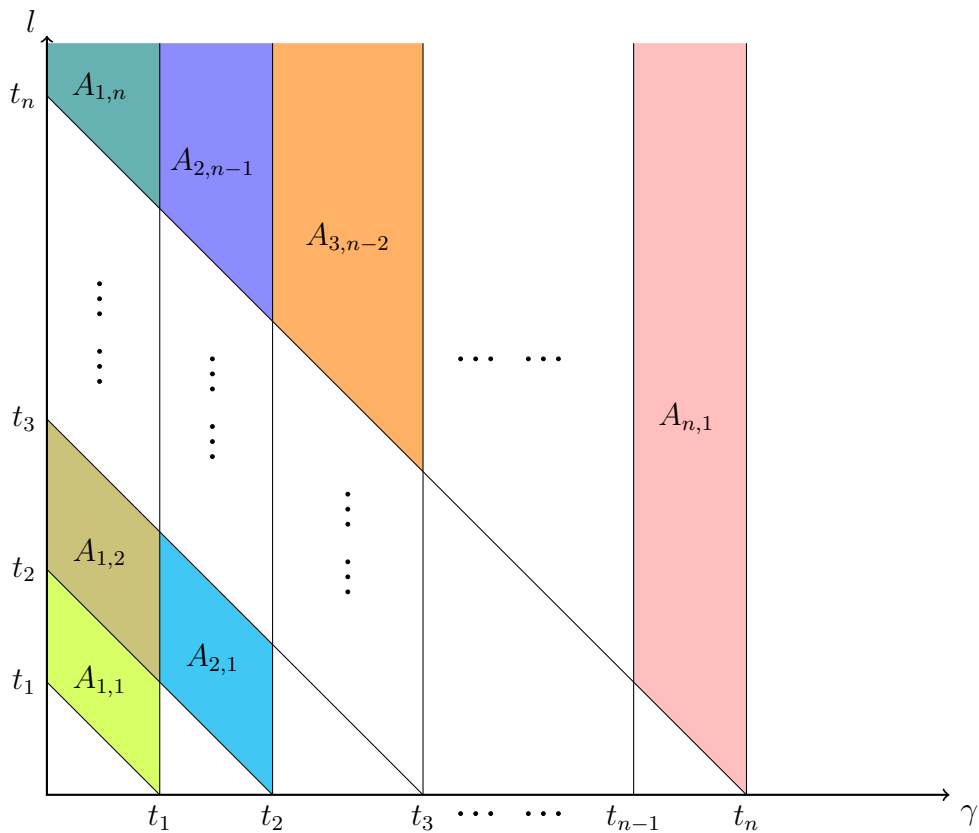
\begin{figure}
\begin{tikzpicture}
\path [fill = lime!60] (1.5,0) -- (1.5,1.5) -- (0,3) -- (0,1.5);
\path [fill = olive!50] (1.5,3.5) -- (1.5,1.5) -- (0,3) -- (0,5);
\path [fill = teal!60] (1.5,7.8) -- (1.5,10) -- (0,10) -- (0,9.3);

\path [fill = cyan!60] (1.5,1.5) -- (1.5,3.5) -- (3,2) -- (3,0);
\path [fill = blue!45] (1.5,7.8) -- (1.5,10) -- (3,10) -- (3,6.3);

\path [fill = orange!60] (3,10) -- (3,6.3) -- (5,4.3) -- (5,10);
\path [fill = pink] (7.8,1.5) -- (9.3,0) -- (9.3,10) -- (7.8,10);

\draw [thick] [<->] (0,10.1) -- (0,0) -- (12,0);
\draw (0,1.5) -- (1.5,0);
\node [left] at (0, 1.5) {$t_{1}$};
\draw (0, 3) -- (3,0);
\node [left] at (0, 3) {$t_{2}$};
\draw (0, 5) -- (5,0);
\node [left] at (0, 5) {$t_{3}$};
\draw (0, 9.3) -- (9.3,0);
\node [left] at (0, 9.3) {$t_{n}$};

\draw (1.5,0) -- (1.5, 10);
\node [below] at (1.5,0) {$t_1$};
\draw (3, 0) -- (3,10);
\node [below] at (3,0) {$t_2$};
\draw (5,0) -- (5, 10);
\node [below] at (5,0) {$t_3$};
\draw (7.8,0) -- (7.8,10);
\node [below] at (7.8,0) {$t_{n-1}$};
\draw (9.3,0) -- (9.3,10);
\node [below] at (9.3,0) {$t_{n}$};

\draw[fill] (5.5,5.8) circle [radius=0.025] ;
\draw[fill] (5.7,5.8) circle [radius=0.025] ;
\draw[fill] (5.9,5.8) circle [radius=0.025] ;

\draw[fill] (6.4,5.8) circle [radius=0.025] ;
\draw[fill] (6.6,5.8) circle [radius=0.025] ;
\draw[fill] (6.8,5.8) circle [radius=0.025] ;

\draw[fill] (5.5,-0.25) circle [radius=0.025] ;
\draw[fill] (5.7,-0.25) circle [radius=0.025] ;
\draw[fill] (5.9,-0.25) circle [radius=0.025] ;

\draw[fill] (6.4,-0.25) circle [radius=0.025] ;
\draw[fill] (6.6,-0.25) circle [radius=0.025] ;
\draw[fill] (6.8,-0.25) circle [radius=0.025] ;

\node at (0.7,1.5) {$A_{1,1}$};
\node at (0.7,3.2) {$A_{1,2}$};
\node at (0.7, 9.4) {$A_{1,n}$};
\draw[fill] (0.7,5.5) circle [radius=0.025] ;
\draw[fill] (0.7,5.7) circle [radius=0.025] ;
\draw[fill] (0.7,5.9) circle [radius=0.025] ;
\draw[fill] (0.7,6.4) circle [radius=0.025] ;
\draw[fill] (0.7,6.6) circle [radius=0.025] ;
\draw[fill] (0.7,6.8) circle [radius=0.025] ;

\node at (2.2, 1.8) {$A_{2,1}$};
\node at (2.2, 8.4) {$A_{2,n-1}$};
\draw[fill] (2.2,4.5) circle [radius=0.025] ;
\draw[fill] (2.2,4.7) circle [radius=0.025] ;
\draw[fill] (2.2,4.9) circle [radius=0.025] ;
\draw[fill] (2.2,5.4) circle [radius=0.025] ;
\draw[fill] (2.2,5.6) circle [radius=0.025] ;
\draw[fill] (2.2,5.8) circle [radius=0.025] ;

\node at (4, 7.4) {$A_{3,n-2}$};
\draw[fill] (4,2.7) circle [radius=0.025] ;
\draw[fill] (4,2.9) circle [radius=0.025] ;
\draw[fill] (4,3.1) circle [radius=0.025] ;
\draw[fill] (4,3.6) circle [radius=0.025] ;
\draw[fill] (4,3.8) circle [radius=0.025] ;
\draw[fill] (4,4) circle [radius=0.025] ;

\node at (8.5,5) {$A_{n,1}$};

\node [above left] at (0,10) {$l$};
\node [below right] at (12,0) {$\gamma$};

\end{tikzpicture}
\caption{Disjoint regions used for the limit of finite dimensional distributions.}
\label{fig:A_i-region-n-timepoints}
\end{figure}

\subsection{Tightness and Homogenization Results}
In this section, we shall summarize our previous considerations about the limiting behavior of our queueing system. This will yield the homogenization results alluded to in the introduction.
The next natural step in establishing our limiting description of 
the family $\{N^{\ep} ; \ep > 0\}$ is a tightness result. Due to the expression \eqref{eq:N-t} for $N^{\ep}$, it is natural to consider this process (restricted on the interval $[0, T]$) as an element of the following space:
\begin{equation}\label{def:D_T-functionspace}
D_T = \left\{f: [0, T] \to \R_+;\,\, f \,\,\text{right-continuous with left limits}\right\} .
\end{equation}
The Borel $\sigma$-field of $D_{T}$ will be denoted by $\mathcal{D}$. Then according to \cite[Proposition 4.2]{FV}, the tightness of $\{N^{\ep} \,; \,\ep > 0\}$ in the space $D_T$ stems from the following criterion.

\begin{proposition}\label{prop:criterion-for-tightness}
Let $X$ and $X_n$, $n \in \mathbb{N}$, be random variables in $(D_T, \mathcal{D})$. Suppose that\\
\quad (1) $\left(X_n (t_1), \ldots , X_n (t_k) \right) \xrightarrow{(\dd)} \left(X (t_1), \ldots , X(t_k) \right)$, for $k$ instants $t_1< \cdots <t_k$,\\
\quad (2) $X$ has jumps of size $\pm 1$, and\\
\quad (3) $X_n$ has integer-valued jumps.\\
Then the sequence $\left\{X_n : \, n\ge 1 \right\}$ is tight.
\end{proposition}

\noindent
Now we state a proposition about the tightness of $\{N^{\ep}; \ep > 0\}$ in the space $D_T$.

\begin{proposition}\label{Prop:tightness-Poi-N-ep}
Let $\{N^{\ep} \,;\, \ep > 0\}$ be the sequence of Poisson processes defined by \eqref{eq:N-t}, which belongs to the space $D_T$. Then, $\PP_Z$-almost surely $\{N^{\ep} \,;\, \ep > 0\}$ is tight.
\end{proposition}

\begin{proof}
We observe that $N^{\ep} (t)$ defined by \eqref{eq:N-t} and $N (t)$ defined by \eqref{def-eq:N-t} 
have intrinsically jumps of size $\pm 1$. In the light of \cref{Cor:convergence-N-ep}, the tightness of the sequence $\{N^{\ep} \,;\, \ep > 0\}$ is a direct consequence of \cref{prop:criterion-for-tightness}.
\end{proof}

We now turn to our main result which is the following quenched limit theorem for the process $N^{\ep}$.

\begin{theorem}\label{thm:limit-N-ep-in-D_T}
Consider an arbitrary time horizon $T> 0$. We assume that \cref{hyp:ergodicity,hyp:arr-times-Gamma,hyp:F-bar-increments} are verified. Recall that the processes $N^{\ep}$ and $N$ are respectively defined by \eqref{eq:N-t} and \eqref{def-eq:N-t}, and the functional space $D_T$ is introduced in \eqref{def:D_T-functionspace}. Then as $\ep \to 0$, $\PP_Z$-almost surely the following limit in distribution holds true in $D_{T}$:
\begin{equation*}
\{N^{\ep} (t) : \, t \in [0, T]\} \xrightarrow{(\dd)} \{N(t) : \, t \in [0, T]\} .
\end{equation*}
\end{theorem}

\begin{proof}
The convergence in law of $N^{\ep}$ follows from \cref{Cor:convergence-N-ep} and \cref{Prop:tightness-Poi-N-ep}.
\end{proof}

\subsection{Limiting Description of Associated Processes}
We now state some consequences of \cref{thm:limit-N-ep-in-D_T} that are of interest in practice. 

We first consider the total accumulated input on the interval $[0, t]$, defined by
\begin{equation}\label{eq:A-ep-input}
A^{\ep} (t) = \iot N^{\ep} (s) \dd s \,.
\end{equation}
The process $A^{\ep}$ is a continuous function on $[0,T]$, due to the fact that $N^{\ep} \in D_T$. In some real-world situations,
this continuous quantity models a stochastic fluid input to a queueing system. The limiting behavior of $A^{\ep}$ is summarized in the following proposition.

\begin{proposition}\label{thm:cont-arrival-limit}
Let $C_T$ be the space of continuous functions from $[0,T]$ to $\R$.
With the same assumptions of \cref{thm:limit-N-ep-in-D_T}, the input process $A^{\ep} (t)$ defined by \eqref{eq:A-ep-input} converges in distribution. More precisely, $\mathbb{P}_Z$-almost surely we have the following limit in law in the space $C_T$ as $\ep \to 0$,
\begin{equation*}
A^{\ep} (t) = \iot N^{\ep} (s) \dd s \, \xrightarrow{(\dd)} \, \iot N (s) \dd s \,,
\end{equation*}
where $N$ is the process given by \eqref{def-eq:N-t}.
\end{proposition}

\begin{proof}
Let $\phi : D_T \to C_T$ be defined by
\begin{equation*}
\left[\phi (f)\right]_t = \iot f(s) \dd s \,, \quad \text{for} \quad f \in D_T \,.
\end{equation*}
It is readily checked that $\phi$ is a continuous function. 
Since $N^{\ep} \xrightarrow{(\dd)} N$ in $D_T$, we get that
\begin{equation*}
A^{\ep} = \phi (N^{\ep}) \, \xrightarrow{(\dd)} \, \phi (N) = A , \quad \text{in the space} \,\, C_T \,.
\end{equation*}
This completes our proof.
\end{proof}

{Another useful corollary of our main 
\cref{thm:limit-N-ep-in-D_T} is the following.}
The quantity $A^{\ep} (t)$ defined by \eqref{eq:A-ep-input} may be treated as a $G_t/G_t/\infty$-input for another single-server queue. {The state of the single server queue, $X^\ep(t)$,
satisfies the following storage equation driven by 
$dA^{\ep}(t) = N^{\ep}(t)dt$},
\begin{equation*}
\dd X^{\ep} (t) = N^{\ep} (t) \dd t - r \mathbbm{1}_{\{X^{\ep} (t) >0\}} \dd t \,, \quad X^{\ep} (0) = 0 \,,
\end{equation*}
where we assume that the server works at constant rate $r$.
Thanks to Skorohod's lemma 
on reflected processes (see e.g \cite[Theorem 6.1]{chen-yao}), the application $N^{\ep} \to X^{\ep}$ is continuous from $D_T$ to $C_T$. Therefore, we obtain a limiting behavior for $X^{\ep}$ as follows:

\begin{proposition}
Let the assumptions of \cref{thm:limit-N-ep-in-D_T} prevail. Then $\PP_Z$-almost surely, $X^{\ep}$ converges in distribution to a process $X$, such that $X$ solves the following equation driven by $N$,
\begin{equation*}
\dd X(t) = N(t) \dd t - r \mathbbm{1}_{\{X(t) > 0\}} \dd t ,
\end{equation*}
with the initial condition $X(0) = 0$.
\end{proposition}


%
%

\section{{Generalizations and Conclusions}}\label{sect:generalization}
We now demonstrate, in Section~\ref{sec:ext}, that our main result in Theorem~\ref{thm:limit-N-ep-in-D_T} can be extended to more general stochastic intensity models that do not have the product form assumed before. We end in Section~\ref{sec:rem} with a summary of the paper and comments on future research directions.

\subsection{Extension}~\label{sec:ext}
Thus far we have assumed that the stochastic intensity model is of product form: $\la (s) \psi (Z_{s/ \ep})$. Our analysis can be applied to a more
general of the stochastic intensity $\Psi : [0,\infty) \times \R^d \to \R_+$, where $\Psi$ is positive, 
and for simplicity assumed to be bounded and uniformly continuous in both variables.
The key to obtaining our limiting results, therefore, lies in the analysis of
\begin{equation}\label{eq:lim-gen-intPsi}
\lim_{\ep \to 0} \frac{1}{t} \iot \Psi (s, Z_{s/\ep}) \dd s .
\end{equation}
{The following argument shows how the method presented in the 
previous sections can be replicated in this general case.}

Consider the Riemann sum of $\Psi$ over $[0, t]$ with uniform partition $0 = t_0 < t_1 < \cdots < t_n = t$; i.e., $t_i = it/n$, for $i = 0, 1, \cdots, n$. We rewrite the integral in \eqref{eq:lim-gen-intPsi} as
\begin{equation*}
\iot \Psi (s, Z_{s/\ep}) \dd s = \sum_{i = 0}^{n-1} \int_{t_i}^{t_{i+1}} \Psi (s, Z_{s/\ep}) \dd s.
\end{equation*}
By the uniform continuity of $\Psi$ in the first variable, we have
as $n \to \infty$ that
\begin{equation*}
\lim_{n\rightarrow \infty}\Bigg| \sum_{i = 0}^{n-1} \int_{t_i}^{t_{i+1}} \Big(\Psi (s, Z_{s/\ep}) - \Psi (t_i, Z_{s/\ep}) \Big) \dd s \Bigg| = 0.
\end{equation*}
{The above procedure essentially 
freezes the time variable $s$ to the discrete epochs $t_i$'s.}
\noindent
Hence, the conclusion that the limit in \eqref{eq:lim-gen-intPsi} exists and is finite results from the following computation:
\begin{eqnarray*}
\lim_{\ep\to 0}\int_{t_i}^{t_{i+1}} \Psi (t_i, Z_{s/\ep}) \dd s 
&=& \lim_{\ep\to 0}(t_{i+1} - t_i)
\frac{1}{t_{i+1} - t_i}\int_{t_i}^{t_{i+1}} \Psi (t_i, Z_{s/\ep}) \dd s
\notag \\
&=& \lim_{\ep \to 0} 
(t_{i+1} - t_i)
\frac{1}{\frac{t_{i+1} - t_i}{\ep}}
\int_{\frac{t_i}{\ep}}^{\frac{t_{i+1}}{\ep}} \Psi (t_i, Z_{s}) \dd s\\
&=& (t_{i+1} - t_i) \E \big[\Psi (t_i, \bar{Z}) \big]
\end{eqnarray*}
where $\bar{Z}$ is the stationary solution corresponding to the process $Z$
and the limit holds $\mathbb{P}_Z$-almost surely.
Hence the limit in \eqref{eq:lim-gen-intPsi} can be consequently written as
\begin{equation*}
\lim_{\ep \to 0} \iot \Psi (s, Z_{s/\ep}) \dd s
= \iot \E \big[ \Psi (s, \bar{Z}) \big] \dd s 
\equiv \int_0^t \left(\int \Psi(s,y) \pi(\dd y) \right) \dd s
\end{equation*}
where $\pi$ is the stationary measure for $Z_\cdot$.

It is then straightforward to show that the following analogue of ~\cref{prop:m-t-of-N-t} holds.
\begin{proposition}~\label{prop:m-t-of-N-t-2}
Let $M^\ep$ and $\{N^{\ep} (t): t \geq 0\}$ be defined
by \eqref{eq:M-sum} and \eqref{eq:N-t}, respectively. Then under the quenched probability 
$\PP_Z$,
$M^\ep$ is a Poisson random measure with mean measure given by
\begin{equation*}
\tilde{\nu}^{\ep} (\dd x, \dd y) = \nu (x, \dd y) \mu^{\ep} (\dd x) \,,
\quad \text{with} \quad
\mu^{\ep} (\dd s) = \Psi (s,Z_{s/\ep}) \dd s \,,
\end{equation*}
where we recall that $\nu$ is introduced in \eqref{eq:def-bar-F}. Furthermore,  we have that for any $t > 0$, 
$N^{\ep} (t)$ is a Poisson random variable with parameter
\begin{equation*}
m^{\ep} (t) = \int_{\{(x,y): x<t< x+y\}} \nu (x, \dd y) \mu^{\ep} (\dd x) = \int_0^t \Psi(s,Z_{s/\epsilon}) \bar F_s(t-s) \dd s.\end{equation*}
\end{proposition}

Consequently, analogous to Theorem~\ref{Thm:limit-m-ep}, we claim that:

\begin{theorem}~\label{thm:limit-N-ep-in-D_T-2}
	Let the assumptions of ~\cref{prop:m-t-of-N-t-2} prevail. We also suppose that ~\cref{hyp:ergodicity,hyp:F-bar-increments} hold without change, while ~\cref{hyp:arr-times-Gamma} holds with the intensity $\{\Psi(s,Z_{s/\epsilon}); s \geq 0\}$. Then under the quenched probability $\mathbb P_Z$, for any $t > 0$ we have almost surely
	\begin{align*}
		\lim_{\ep \to 0} m^{\ep} (t) = \bar m(t),
	\end{align*}
	where 
	\begin{align*}
		\bar m(t) = \int_0^t \mathbb E[\Psi(s,\bar Z)] \bar F_s(t-s) \dd s.
	\end{align*}
\end{theorem}
The asymptotic convergence result in ~\cref{thm:limit-N-ep-in-D_T}, therefore, can be generalized so that the corresponding limit process $\{N(t) : t \in [0,T]\}$ is a Poisson point process with mean intensity $\bar{m} (t)$. More importantly, observe that yet again, there is a time-scale separation in the homogenization limit.





\subsection{Conclusion and Perspectives}~\label{sec:rem}
Our primary results in ~\cref{thm:limit-N-ep-in-D_T,thm:limit-N-ep-in-D_T-2} demonstrate that the state of a $G_t/G_t/\infty$ queue in a random fast oscillatory environment is closely approximated by that of a $M_t/G_t/\infty$ queue with nonhomogeneous Poisson traffic, in the homogenization limit. More precisely, the rapid fluctuations of the stochastic intensity are averaged out in the limit. These results are very useful for performance analysis since much is known about the properties of the $M_t/G/\infty$ queue, which is also a much simpler object to simulate. Our model assumes a two time-scale structure, wherein time-of-day effects in the traffic intensity are modeled by a smoothly varying function, and stochastic fluctuations are modeled by a strongly ergodic stochastic process. We also assume a very general model of time-varying service wherein the service time distribution itself depends on the arrival epoch. A crucial insight that emerges from our analysis is the fact that we do not require exponential ergodicity of the stochastic environment, though we do require a strong sense of ergodicity to be satisfied. Furthermore, our analysis permits both light- and heavy-tailed service time distributions. Therefore, for a rather broad range of infinite server queueing models, the system state is well approximated by a much simpler $M_t/G/\infty$ queue.  

Of course, these insights are greatly facilitated by the fact that we study an infinite server queue, allowing us to leverage the properties of Poisson point processes. In this setting we anticipate proving a functional central limit theorem (FCLT) for the system state in the homogenization limit, complementing ~\cref{thm:limit-N-ep-in-D_T,thm:limit-N-ep-in-D_T-2} with a rate of convergence. In a stationary setting where the arrival intensity $\lambda(\cdot)$ is a constant, it is well known that a FCLT holds and that the approximating process is O-U; see \cite[Section 3]{HvLM}. In our setting, however this analysis is complicated by the fact that the centering is by a time-varying function $\bar m(\cdot)$, and the analysis appears to require some further technical development, that is outside the scope of this paper. Second, while our results are in the homogenization limit as $\epsilon \to 0$, it would be interesting to consider the large time behavior of the process $N^\epsilon$ for a fixed $\epsilon > 0$, when $\lambda(\cdot)$ is a constant. Given the more general setting we are studying, this type of result would expand on the results in~\cite[Section 3]{HvLM}. It should also be noted that in~\cite{HvLM} the traffic model is a special stationary DSPP where the stochastic intensity process is constructed by sampling an i.i.d. stochastic process. That paper establishes a rather interesting ``trichotomy'' result, in particular showing that if the sampling is ``rapid'' then the traffic process is Poisson-like, reflecting an averaging effect; on the other hand, they also find that if the sampling is ``slow,'' then the over-dispersed nature of the DSPP is maintained in the limit, and consequently the limit process is {\it not} a $M_t/G/\infty$ queue. The small $\epsilon$ setting in this paper is, in a sense, a more general ``rapid'' sampling procedure. One of the surprises of our results is the fact that we are able to recover the Poisson-like structure in the limit, even with heavy-tailed service and polynomial ergodicity of the underlying stochastic intensity. {We do not, however, have a result that parallels the ``slow'' sampling result in~\cite{HvLM}. This suggests that there are regimes where time-scale separation between the time-of-day effects and the stochastic intensity are not manifested in the limit.} This appears to require a more refined CLT-type analysis. We will address this interesting phenomenon in a future paper.

\bibliographystyle{plain}
\bibliography{references}

\end{document}